\documentclass[a4paper,11pt,twoside,dvips]{amsart}

\usepackage{mathrsfs}
\usepackage{float}
\usepackage{color}
\usepackage{amssymb}
\usepackage{amsfonts}
\usepackage{amsthm}
\usepackage{amsmath}
\usepackage{newlfont}
\usepackage{graphicx}
\usepackage{graphics}
\usepackage{graphpap}
\usepackage[english, francais]{babel}
\usepackage{srcltx}
\input xy
\xyoption{all}

\newtheorem{theorem}{Theorem}
\newtheorem{lemma}{Lemma}
\newtheorem{proposition}{Proposition}
\newtheorem{remark}{Remark}

\begin{document}

\title{Completeness is determined by any non-algebraic trajectory}

\selectlanguage{english}

\author{Alvaro Bustinduy}
\address{Departamento de Ingenier{\'\i}a Industrial \newline
         \indent Escuela Polit{\'e}cnica Superior \newline
         \indent Universidad Antonio de Nebrija \newline
         \indent C/ Pirineos 55, 28040 Madrid. Spain}
\email{abustind@nebrija.es}
\author{Luis Giraldo}
\address{Departamento de Geometr{\'\i}a y Topolog{\'\i}a \newline
         \indent Facultad de Ciencias Matem\'aticas \newline
         \indent Universidad Complutense de Madrid \newline
         \indent Plaza de Ciencias 3, 28040 Madrid. Spain}
\email{luis.giraldo@mat.ucm.es}

\thanks{2000 {\it Mathematics Subject Classification.} Primary 32M25;
Secondary 32L30, 32S65}
\thanks {{\it Key words and phrases.} Complete vector field,
complex orbit, holomorphic foliation}
\thanks{Supported by Spanish MICINN
projects MTM2010-15481, MTM2011-26674-C02-02.}

\begin{abstract}
\selectlanguage{english} It is proved that any polynomial vector
field in two complex variables which is complete on a
non-algebraic trajectory is complete.
\end{abstract}

\dedicatory{Dedicated to the memory of Marco Brunella}

\selectlanguage{english}

\maketitle \markright{COMPLETE VECTOR FIELDS}

\section{Introduction and Statement of Results}

Let $X$ be a holomorphic vector field on $\mathbb{C}^2$. For any
$z\in\mathbb{C}^{2}$, the local solution $\varphi_{z}(T)$ of the
associated ordinary differential equation $dz/dT=X(z(T))$ with the
initial condition $z(0)=z\in\mathbb{C}^{2}$ can be extended by
analytic continuation along paths in $\mathbb{C}$, to a maximal
domain $\Omega_z$, which may not be an open set of $\mathbb{C}$,
but rather a Riemann domain over $\mathbb{C}$. The map
${\varphi}_z:\Omega_z \rightarrow \mathbb{C}^2$ is the {\em
solution} of $X$ through $z$, and its image $\varphi_z(\Omega_z)$,
that will be denoted by $C_z$ (or $L_{z}$, $R_z$, $S_z$), is the
{\em trajectory} of $X$ through $z$.

The vector field $X$ is complete on $C_z$ if
$\Omega_z=\mathbb{C}$, and $X$ is {\em complete} if it is complete
on $C_z$, for every $z\in\mathbb{C}^2$. Each trajectory $C_z$ on
which $X$ is complete (complete trajectory) is defined by an
abstract Riemann surface uniformized by $\mathbb{C}$, and  by the
maximum principle, analytically isomorphic to $\mathbb{C}$ or
$\mathbb{C}^{\ast}$.

Extrinsically, the topology of a trajectory
 can be very complicated. The simplest
trajectories from this point of view are the analytic ones. One
says that the trajectory $C_{z}$ is \em analytic \em if it is
contained in an analytic curve in $\mathbb{C}^{2}$ (but not
necessarily equal to it, due to the possible presence of
singularities). Otherwise $C_{z}$ is a \em
non-analytic  \em trajectory.

An interesting remark (due to R. Moussu) is that two vector fields
with a common non-analytic trajectory have to be collinear in any
point. In this sense a non-analytic trajectory determines the
vector field up to multiplication by a nonvanishing holomorphic
function.

In this work we will consider \em polynomial vector fields with at
most isolated zeros. \em The above remark for two polynomial
vector fields can be restated. For a trajectory, it is enough not
to be contained in an algebraic curve (that is, to be a \em
non-algebraic \em trajectory) to determine the vector field up to
multiplication by a nonzero constant.

In \cite{Brunella-topology2}, M. Brunella studied foliations in
$\mathbb{C}^{2}$ given by polynomial vector fields with a
trajectory containing a planar isolated end  (proper Riemann
sub-surface isomorphic to $\{z: r<||z||\leq{1}\}$, where $r\in
[0,1)$), properly embedded in $\mathbb{C}^2$ and whose closure in
$\mathbb{CP}^{2}$ contains the line at infinity. He proved that
these foliations can be determined in terms of a polynomial whose
generic fiber is of type $\mathbb{C}$ or $\mathbb{C}^{\ast}$ and
transversal to the foliation. As remarkable corollary, he obtained
that if the trajectory is a non-algebraic analytic plane, the
foliation is given by the constant vector field after an analytic
automorphism. So the trajectory in this case is determining the
completeness of the vector field up to multiplication by a nowhere
vanishing function.

Then if one attends to the completeness of a non-algebraic
trajectory (not necessarily analytic) the following natural
question arises \cite[Question 5.1]{Bustinduy-indiana}:

\vspace{0.25cm}

\em \noindent {\bf Question 1.} If $X$ is a polynomial vector
field in $\mathbb{C}^2$ with the property of being complete on a
single non-algebraic trajectory, is it complete? \em

\vspace{0.25cm}

The main result of this work says that Question 1 has an affirmative answer:
\begin{theorem}
\label{principal} Let us consider a polynomial vector field $X$ on
$\mathbb{C}^2$  which is
complete on a non-algebraic trajectory. Then $X$ is complete.
\end{theorem}

Note that our theorem implies that any entire solution of a polynomial
vector field can be determined up to an algebraic automorphism of
$\mathbb{C}^2$. As the vector field is complete, the solution must
correspond to one of the vector fields of the Brunella's
classification in \cite{Brunella-topology} after a polynomial
automorphism.

It could be very interesting to study if a non-analytic trajectory of a (non-polynomial) holomorphic vector field
determines the completeness of the vector field.

\subsection*{About the proof of Theorem~\ref{principal}}
For the sake of completeness, throughout the paper we include
some definitions and results taken from
Brunella's papers \cite{Brunella-impa}, \cite{Brunella-ensp} and
\cite{Brunella-topology}. Let us begin by
recalling some definitions. Consider the foliation
$\mathcal{F}$ generated by $X$ on $\mathbb{C}^2$ extended to
$\mathbb{CP}^2$. According to Seidenberg's Theorem, the minimal
resolution of $\mathcal{F}$ is a new foliation
$\tilde{\mathcal{F}}$ defined on a rational surface $M$ after
pulling back $\mathcal{F}$ by a birational morphism $\pi: M
\rightarrow \mathbb{CP}^2$, that is a finite composition of
blowing ups. Along with this resolution one has: 1) The Zariski
open set $U=\pi^{-1}(\mathbb{C}^2)$ of $M$, over which $X$ can be
lifted to a holomorphic vector field $\tilde{X}$, 2) \em the
exceptional divisor \em $E$ of $U$, and 3) the \em
divisor at infinity \em
$$D=M\setminus U = {\pi}^{-1}(\mathbb{CP}^2\setminus\mathbb{C}^2)=
{\pi}^{-1}(L_{\infty}),$$ that is a tree of smooth rational
curves. The vector field $\tilde{X}$ can be extended to $M$,
although it may have poles along one or more components of $D$.
Let us still denote this extension by $\tilde{X}$.
As only singularities of the foliation in $\mathbb{C}^2$
are blown up, and they are in  the zero set of $X$, the vector field $\tilde{X}$ is holomorphic
on the full $U$ and it has the complete trajectory $\tilde{C}_z$
defined by $\pi^{-1}(C_z)$.

Therefore the reduced foliation $\tilde{\mathcal{F}}$ has at least
one tangent entire curve: the one defined by $\tilde{C}_z$, which
is Zariski dense in $M$. It implies  that \em the Kodaira
dimension \em $\textnormal{kod}(\tilde{\mathcal{F}})$ of
$\tilde{\mathcal{F}}$ is $1$ or $0$ \cite[\S IV]{Mc} (\em see \em
also \cite[p.\,131]{Brunella-impa}).

In case $\textnormal{kod}(\tilde{\mathcal{F}})=1$,
\cite{Brunella-topology} allows to conclude that
$\tilde{\mathcal{F}}$ is a Riccati foliation adapted to a
fibration $g: M \to \mathbb{P}^1$, whose projection to
$\mathbb{C}^2$ by $\pi$ defines a rational function $R$ of type
$\mathbb{C}$ or $\mathbb{C}^{\ast}$. We can apply the
study of \cite{Bustinduy-indiana} although $R$ is not a polynomial
(\em see \em also \cite{Bustinduy-ijm})  and deduce the
completeness of $X$. We will analyze this case in $\S2$.

In case $\textnormal{kod}(\tilde{\mathcal{F}})=0$, we know that
$\tilde{\mathcal{F}}$ is generated by a vector field on a smooth
compact projective surface $S$, up to contractions of
$\tilde{\mathcal{F}}$-invariant curves and covering maps
\cite{Brunella-topology}. However, we need to go a bit further to
know if these models restrict to our open $U$ a complete vector
field. This is accomplished via the description of the irreducible
components of $D\cup E$ that are not
$\tilde{\mathcal{F}}$-invariant. When $S$ is rational, we show
that in fact $D\cup E$ must be invariant if
$\tilde{\mathcal{F}}$ is not Riccati with respect to a
fibration $g: M \to \mathbb{P}^1$ that is projected to
$\mathbb{C}^2$ by $\pi$ as a rational function $R$ of type
$\mathbb{C}$ or $\mathbb{C}^{\ast}$. For the remaining cases,
i.e. when $S$ is a $\mathbb{CP}^1$-bundle over an elliptic curve
or a complex $2$-torus, we prove that $D\cup
E$ is always invariant by $\tilde{\mathcal{F}}$. For the proof
of this last fact we will consider $S$ as a differential manifold
with a certain Riemannian metric. It will enable us to compute the
distance from the complete trajectory to a compact set containing
the components of $D\cup E$ that are not
$\tilde{\mathcal{F}}$-invariant. As a consequence of the
discussion above one obtains that the lifted of $\tilde{X}$ by a
certain covering map can be decomposed in the product of a
complete rational vector field by a second integral of it. It
allows us to conclude that the projection $\pi_{\ast} \tilde{X}$
restricted to $U$ i.e. $X$ must be complete. We will analyze this
case in $\S3$.

Finally, we point out that \cite{Bustinduy-indiana}, \cite{Bustinduy-jde}
imply that Question 1 has an affirmative answer for a non-algebraic analytic trajectory.
In those works, Brunella's results \cite{Brunella-topology2} are used as the main tool.
The proof of our theorem is mainly based on Brunella's approach to the classification complete polynomial vector fields in the plane
\cite{Brunella-topology}, since they can be applied to the
foliation $\mathcal{F}$ although $X$ could be in principle
not complete. Theorem~\ref{principal} is not only the generalization
of the previous results  mentioned above (\cite{Bustinduy-indiana}, \cite{Bustinduy-jde}), but its proof also implies them.

\subsection*{Acknowledgements}  This article is dedicated to the memory of Marco Brunella, whose deep mathematical work about complete vector fields has been crucial to
obtain our results. We also appreciate a lot his generosity in several mathematical conversations during these last years.
We received the sad notice of his death when we were preparing the revised version of this article.

Finally, we also want to thank the referee for his suggestions that have improved
this paper a lot.

\section{\bf {\textnormal{kod}($\mathcal{F}$)=1}}
According to \cite[\S IV]{Mc} the absence of a first integral
implies that $\tilde{\mathcal{F}}$ is \em a Ricatti or a Turbulent
foliation, \em that is to say, the existence of a fibration
$$g: M\to B$$
whose generic fibre is a rational curve or an elliptic
curve transverse to $\tilde{\mathcal{F}}$, respectively. \em Remark
that $B$ is $\mathbb{CP}^1$  since $M$ is a rational surface. \em

\subsection{Nef models and Canonical models
\cite[\S III]{Mc}, \cite[\S\,4]{Brunella-ensp}, \cite[\S\,3]{Brunella-topology}} $\quad$

\noindent \em Existence of a nef model. \em  As $\tilde{\mathcal{F}}$ is not
a rational fibration it has a model \em $\hat{\mathcal{F}}$ which
is \em reduced and nef. More concretely, after a contraction
$s:M\to \hat{M}$ of the $\tilde{\mathcal{F}}$-invariant rational
curves on $M$ over which the canonical bundle
$K_{\tilde{\mathcal{F}}}$ has negative degree one obtains (\em see
\em \cite[\S\,4]{Brunella-ensp}, \cite[\S\,3]{Brunella-topology}):

\begin{enumerate}

\item [1)] A new surface $\hat{M}$, maybe with \em cyclic quotient singularities; \em and

\item [2)] A reduced foliation
$\hat{\mathcal{F}}=s_{\ast}\tilde{\mathcal{F}}$ on $\hat{M}$ such
that its canonical $\mathbb{Q}$\,-bundle $K_{\hat{\mathcal{F}}}$
is \em nef \em (i.e. $K_{\hat{\mathcal{F}}}\cdot C \geq{0}$  for
any curve $C\subset\hat{M}$).

\end{enumerate}

Recall that a cyclic quotient singularity $p$ of $\hat{M}$ is
locally defined by  $\mathbb{B}^{2}/{\Gamma}_{k,h}$ where
$\mathbb{B}^2\subset\mathbb{C}^2$ is the unit ball and
${\Gamma}_{k,h}$ is the cyclic group generated by a map of the
form $(z,w)\to (e^{\frac{2\pi i }{k}}z,e^{\frac{2\pi i}{k}h}w)$ with
$k,h$ positive coprime integers such that $0 < h < k$. \em These
singularities of $\hat{M}$ are not singularities of
$\hat{\mathcal{F}}$. \em That is, the foliation can be lifted
locally to $\mathbb{B}^2\setminus\{(0, 0)\}$ and extended to a
foliation on $\mathbb{B}^2$ with a non-vanishing associated vector
field.

\begin{remark}\label{ciclicas} \em
The possible cyclic singularities of $\hat{M}$ are in the image of
the exceptional divisor of $s$. Any rational curve $C_0$ of that
divisor is $\tilde{\mathcal{F}}$-invariant, it has an unique
singularity $p$ of the foliation of type $d(x^ny^m)$ with $n,m\in
\mathbb{N}^{+}$, where $C_0=\{y=0\}$, and it may also contain one
cyclic quotient singularity $q$ of order $m$ (regular point if
$m=1$). After contracting $C_0$ by $s$ (since $C_{0}^{2} =-n/m$)
we obtain a new quotient singularity of order $n$ (regular if
$n=1$) \cite[pp.\,443\,-444]{Brunella-topology}. \em
\end{remark}

\noindent \em Existence of a minimal model. \em After possibly
additional contractions on $\hat{M}$ of rational curves,
$q:\hat{M}\to N$, one obtains a reduced foliation
$\mathcal{H}=q_{\ast}\hat{\mathcal{F}}$ (birational to
$\tilde{\mathcal{F}}$) on a surface $N$ regular on the (cyclic
quotient) singularities of $N$ whose canonical bundle
$K_{\mathcal{H}}$ is nef and such that it verifies this property:
if $K_{\mathcal{H}}\cdot C= 0\Rightarrow C^2 \geq{0} $ for any
curve $C\subset N$. It is important to note that we can assume
that $q$ is given by contractions of \em curves which are
invariant by the foliation: \em if $C$ is not
$\hat{\mathcal{F}}$-invariant it follows from the formula
$(K_{\hat{\mathcal{F}}}+C)\cdot C\geq {0}$ \cite[\S\,2]{Brunella-ensp}
that $K_{\hat{\mathcal{F}}}\cdot C= 0\Rightarrow C^2 \geq{0}$. This
model is the \em minimal model \em of $\tilde{\mathcal{F}}$.

$$
\xymatrix{   M \ar[d]^{s}  \ar[r]^{\pi}  &  \mathbb{CP}^2   &  \\
                 \hat{M}   \ar[r]_{q}             &    N     }
$$

\begin{remark}\label{nef} \em
In general the minimal model of $\tilde{\mathcal{F}}$ \em is not
unique. \em However if we have another minimal model
$\mathcal{H}'$ of $\tilde{\mathcal{F}}$ defined on $N'$ and
$p:N\to N'$ is an algebraic map defined everywhere with
$p_{\ast}\mathcal{H}=\mathcal{H}'$ then $p$ is an isomorphism
\cite[Lemma III.3.1]{Mc}. \em
\end{remark}

\begin{remark}\label{contraction} \em
As $s$ and $q$ are given by contractions of rational curves which
are invariant by the foliation neither $\tilde{C}_z$ meets the
exceptional divisor of $s$ nor $s(\tilde{C}_z)$ meets the
exceptional divisor of $q$. It implies that there must be  a
parabolic leaf of $\mathcal{H}$: the
leaf that contains the Riemman surface $q(s(\tilde{C}_z))$ that
supports the complete vector field $q_{\ast}(s_{\ast}(
\tilde{X}_{\mid \tilde{C}_z}))$. \em
\end{remark}

\subsection{Turbulent case}
When $X$ is complete the case of a Turbulent $\tilde{\mathcal{F}}$
can be excluded as it is proved in
\cite[Lemma1]{Brunella-topology}. We now prove that it continues
being valid in a more general situation.

\begin{lemma}
\label{lema1} $\tilde{\mathcal{F}}$ is not a Turbulent foliation
\end{lemma}

\begin{proof}
Suppose  that $\tilde{\mathcal{F}}$ is Turbulent. The description
of models around each fibre of $g$  after a birational morphism
$\alpha: M\rightarrow M^{\ast}$ is known
\cite[\S\,7]{Brunella-ensp}. The resulted foliation
$\mathcal{G}=\alpha_{\ast}\tilde{\mathcal{F}}$ on $M^{\ast}$ is
regular on the (cyclic quotient) singularities of $M^{\ast}$, it
is Turbulent with respect to $\bar{g}=g\circ {\alpha}^{-1}$, and
each fiber of $\bar{g}$ is of one of the following classes:

\noindent $(a)$ (resp. $(d)$): the fibre is smooth elliptic, transversal (resp. tangent) to
$\mathcal{G}$ and may be multiple.

\noindent $(b)$ (resp. $(e)$): the fibre is rational with three quotient singularities of orders
$k_1$, $k_2$ and $k_3$ satisfying $\frac{1}{k_1}+\frac{1}{k_2}+\frac{1}{k_3}=1$, transversal (resp. tangent)
to $\mathcal{G}$ and of multiplicity $3$, $4$ or $6$.

\noindent $(c)$ (resp. $(f)$): the fibre is rational with four quotient singularities of order $2$;
transversal (resp. tangent) to $\mathcal{G}$ and of multiplicity $2$.

We will call classes $(a)$, $(b)$ and $(c)$ (resp. $(d)$, $(e)$
and $(f)$) as \em transversal fibres (resp. tangent fibres) of
$\bar{g}$. \em

For any leaf $L$ of $\mathcal{G}$ outside tangent fibres of
$\bar{g}$, ${\bar{g}}_{\mid L}:L\to B_{0}$, with
$B_{0}$ defined as $B$ minus the points over tangent fibres of $\bar{g}$, is
a regular covering (in orbifold's sense). The orbifold structure
in $B_{0}$ is the natural structure inherited from the orbifold
structure on $B$ induced by (the local models of) $\bar{g}$
\cite[\S\,7]{Brunella-ensp}.

\noindent \em Claim 1: There must be at least one tangent fibre
$G_{0}$ of $\bar{g}$. \em

We suppose that all the fibres are transversal and obtain a
contradiction. Since $B_{0}=B=\mathbb{CP}^1$, the orbifold
universal covering of any leaf $L$, $\tilde{L}$, is equal to the
one of $B$, $\tilde{B}$.

Let us suppose that $\tilde{B}$ is
$\mathbb{C}$ or $\mathbb{CP}^{1}$. By pulling back sections of
$K_{B}$ under $\bar{g}$ we obtain sections of $K_{\mathcal{G}}$.
We can in this way compute $K_{\mathcal{G}}$
and obtain that
$deg(\bar{g}_{\ast}K_{\mathcal{G}})=-\chi_{orb}(B)$ (see,
\cite[\S\,7]{Brunella-ensp}). On the other hand since
$\textnormal{kod}(\tilde{\mathcal{F}})=\textnormal{kod}(\mathcal{G})=1$
then $deg(\bar{g}_{\ast}K_{\mathcal{G}})>0$. It follows that $\chi_{orb}(B)<0$, what is impossible if $B$
is parabolic (see Appendix E, Lemma E.4, in \cite{Milnor}). Thus $\tilde{B}$ is a disk.

As all the leaves of $\mathcal{G}$ are
hyperbolic and the singularities are isolated (in fact
$\mathcal{G}$ is regular), $K_{\mathcal{G}}$  is nef \cite[Remark
8.8]{Brunella-springer}. Moreover, it is clear that
$K_{\mathcal{G}}\cdot C= 0\Rightarrow C^2 \geq{0}$: If $C$ is not
$\mathcal{G}$-invariant it follows from the formula
$(K_{\mathcal{G}}+C)\cdot C\geq {0}$ \cite[\S\,2]{Brunella-ensp}.
If $C$ is $\mathcal{G}$-invariant the Camacho-Sad Formula
\cite[\S\,2]{Brunella-ensp} implies that $C^{2}=0$ because
$\mathcal{G}$ is regular on $C$. Therefore $\mathcal{G}$ is a
minimal model of $\tilde{\mathcal{F}}$. But then it has necessarily a parabolic leaf (Remarks $2$ and $3$), which is a contradiction.

 \noindent \em Claim 2. If there is an irreducible component $D_{1}$ of
$D\cup E$ that is not
$\tilde{\mathcal{F}}$-\,invariant and which is not contained in
any fiber of $g$, then $D_{1}\subset \{\tilde{X}=0\}$. \em

It is important to note that the strict transform of $D_{1}$ by
$\alpha$, that we also denote by $D_{1}$, is a rational
curve. Otherwise it is a point with infinitely many punctured
disks invariant by $\mathcal{G}$ through it and then a singularity
of $\mathcal{G}$, which is not possible. Hence
$D_{1}\cap G_{0}\neq \emptyset$. Let us denote by
$J$ the leaf of $\mathcal{G}$ that defines the non
algebraic component of $\alpha(\tilde{C}_z)$. There is at least
one accumulation point of $J$ on ${G}_{0}$
because $\bar{g}(J)= g(\tilde{C}_z)$  is
$\mathbb{C}$ or $\mathbb{C}^{\ast}$. It must be a regular point of
the foliation by the absence of singularities of the foliation on
tangent fibers. Thus $J$ must accumulate on
${G}_{0}$. It implies that $\tilde{C}_z\cap D_{1}\neq
\emptyset$ and then $D_{1}\subset \{\tilde{X}=0\}$ by the completeness
of $\tilde{X}_{\mid \tilde{C}_z}$.

Let us take the generic fiber $G$ of $g$, which is transverse to
$\tilde{\mathcal{F}}$. Obviously, $D\cap G\neq \emptyset$. In the
contrary case we have an elliptic curve contained in
$\mathbb{C}^2$, which is impossible (maximum principle). Among the
irreducible components of $D$ cutting $G$ at least one, say
$D_{2}$, is $\tilde{\mathcal{F}}$-\,invariant. Otherwise
$\tilde{X}$ would be holomorphic in a neighborhood of $G$ and it
vanishes on at least one component of $D$ transversal to $G$, what
implies that $\tilde{X}$ is identically zero by Claim 2. The
existence of ${D}_{2}$ is enough to construct a rational integral
for a Turbulent $\tilde{\mathcal{F}}$ as it can be seen in
\cite[p.438]{Brunella-topology}.
\end{proof}

\subsection{Ricatti case}

\begin{lemma}\label{lema2}
$g_{\mid U}$ is projected by $\pi$ as a rational function $R$ of
type $\mathbb{C}$ or $\mathbb{C}^{\ast}$. Moreover, $\tilde{\mathcal{F}}$
is $R$\,-complete.
\end{lemma}
\begin{proof}
Up to contraction of rational curves inside fibers of $g$, which
can produce cyclic quotient singularities of the surface but on
which the foliation is always regular, one has that there are five
possible models for the fibers of $g$ \cite[\S\,7]{Brunella-ensp},
\cite[p.\,439]{Brunella-topology}. Let $L_{0}$ be the leaf of the
foliation defined by $\tilde{C}_z$. One can conclude that the
orbifold universal covering $\tilde{L}_{0}$ of $L_{0}$ is equal to the one of
$B_{0}$, $\tilde{B}_{0}$, where
$B_{0}$ is defined as $\mathbb{CP}^{1}$ minus the points over tangent fibres of $g$
with the natural orbifold structure
inherited from the orbifold structure on $\mathbb{CP}^1$ induced
by (the local models of) $g$. Since $X$ is complete on $C_z$,
$\tilde{L}_{0}$ is biholomorphic to $\mathbb{C}$ and then $L_{0}$ is parabolic.
This fact along
with $\textnormal{kod}(\tilde{\mathcal{F}})=1$ implies by
\cite[Lemma 2]{Brunella-topology} that there must be at least one
fibre ${G}_{0}$ tangent to the foliation of class:

\noindent $(d)$: the fibre is rational with  two saddle-nodes of
the same multiplicity $m$, with strong separatrices inside the
fibre, or of class

\noindent $(e)$: the fibre is rational with two quotient
singularities of order $2$, and a saddle-node of multiplicity $l$,
with strong separatrix inside the fibre.

Firstly one observes that there are irreducible components of
$D\cup E$ that are not contained in any fiber of $g$. Let us take
the generic fiber $G$ of $g$, which is transverse to
$\tilde{\mathcal{F}}$. Obviously, $D \cap G\neq \emptyset$. In the
contrary case we have a rational curve contained in
$\mathbb{C}^2$, which is impossible (maximum principle).

Let $D_{1}$ be one of these components.
Then $D_{1}\cap G_{0}\neq \emptyset$ and there is at
least one accumulation point of $\tilde{C}_z$ on
${G}_{0}$, say $p$, because $g(\tilde{C}_z)$ is
$\mathbb{C}$ or $\mathbb{C}^{\ast}$. If $p$ is a regular point of
the foliation, $\tilde{C}_z$ must accumulate on ${G}_{0}$.
If $p$ is singular, it is a saddle-node with strong separatrix
defined by ${G}_{0}$, and therefore $\tilde{C}_z$ must
also accumulate on all ${G}_{0}$ \cite{MR}, in particular
in the other saddle node if it exists. There are two
possibilities:

\noindent $(i)$ \em If $D_{1}$ is $\tilde{\mathcal{F}}$-invariant, $D_{1}$
is not in the divisor of poles of $\tilde{X}$. \em  Otherwise,
$D_{1}\cap {G}_{0}\neq \emptyset$ is a saddle-node $q$.  Let us
take the rational section $\omega$ of $K_{\tilde{\mathcal{F}}}$
dual to $\tilde{X}$ that restricts to $\tilde{C}_z$ as the \em
differential of times \em given by the flow of $\tilde{X}$ on
$\tilde{C}_z$. One can construct a path $\gamma:(0,\epsilon]\to
\tilde{C}_z$, with $\epsilon\in \mathbb{R}^{+}$ and $\gamma(t) \to
q $ as $t\to 0$, such that $\int _ {\gamma} \omega$ is finite (\em
see \em \cite[proof of Lemma 3]{Brunella-topology}), which
contradicts the completeness of $\tilde{X}$ on $\tilde{C}_z$.

\noindent $(ii)$ \em If $D_{1}$ is not
$\tilde{\mathcal{F}}$-invariant, necessarily $\tilde{C}_z\cap
D_{1}\neq \emptyset$ and $D_{1}\subset \{\tilde{X}=0\}$. \em
Otherwise, as $\tilde{C}_z\cap D_{1} =\emptyset$ one has that
$D_{1}\cap G_{0}\neq \emptyset$ is a saddle-node with $D_{1}$
defining its weak separatrix, which is
$\tilde{\mathcal{F}}$-invariant
\cite[Lemma 11]{Brunella-topology2}.

It follows from $(i)$ and
$(ii)$ that $\tilde{X}$ is holomorphic in a neighborhood of
$G$, what implies as in above lemma that $(ii)$ does not
really occurs. Thus $D_{1}$ is $\tilde{\mathcal{F}}$-invariant.

Therefore $D$ must cut $G$ at one or two points, and the
projection $R$ of $g_{\mid U}$ via $\pi$ is of type $\mathbb{C}$
or $\mathbb{C}^{\ast}$. Moreover, the invariancy of the components
of $D\cup E$ which are not contained in fibers of $g$ implies that
generically $R$ is a fibration trivialized by the leaves of
$\tilde{\mathcal{F}}$, and then $\tilde{\mathcal{F}}$ is
$R$-complete.
\end{proof}

We will study the two possibilities after the previous lemma.

\subsection{$R$ of type $\mathbb{C}$}

By Suzuki (\em see \em \cite{Suzuki-japonesa}) we may assume that
$R=x$, up to a polynomial automorphism. Hence $\mathcal{F}$ is a
Riccati foliation adapted to $x$ and $X$ is a complete vector
field of the form
\begin{equation}
\label{tipoc}
Cx^N\frac{\partial}{\partial x}
+[A(x)y+B(x)]\frac{\partial}{\partial y},\,\,
\end{equation}
with $C\in\mathbb{C}$, $N=0$, $1$ and $A$, $B\in\mathbb{C}[x]$
(\em see \em \cite[Proposition 4.2]{Bustinduy-indiana}).

\subsection{$R$ of type $\mathbb{C}^{\ast}$}
By Suzuki (\em see \em \cite{Suzuki-anales}) we may assume that
$$R=x^ {m}(x^{\ell}y+p(x))^{n},$$
where $m\in\mathbb{N}^\ast$, $n\in\mathbb{Z}^{\ast}$, with
$(m,n)=1$, $\ell\in\mathbb{N}$, $p\in\mathbb{C}[x]$ of degree $<
\ell$ with $p(0)\neq{0}$ if $\ell>0$ or $p(x)\equiv{0}$ if
$\ell=0$, up to a polynomial automorphism.

\subsubsection*{New coordinates} According to relations
$ x=u^n\,\,\,\,\textnormal{and}\,\,\,\,x^{\ell}y+p(x)=v \,u^{-m}$
it is enough to take the rational map $H$ from $u\neq {0}$ to
$x\neq{0}$ defined by
\begin{equation}\label{relaciones}
(u,v)\mapsto (x,y)=(u^n, {u^{-(m+n\ell)}} [v-u^m p(u^n)])
\end{equation}
in order to get $R\circ H(u,v)=v^n$.

Although $R$ is not necessarily a polynomial ($n\in\mathbb{Z}$),
it is a consequence of the proof of \cite[Proposition 3.2]
{Bustinduy-indiana} that $H^{\ast}\mathcal{F}$ is a Riccati
foliation adapted to $v^n$ having $u=0$ as invariant line. Thus

\begin{equation} \label{hest}
\begin{split}
H^{\ast} X = &  u^{k}\cdot Z \\
           =& u^{k} \cdot \left\{a(v)u \frac{\partial}{\partial u} +
c(v)\frac{\partial}{\partial v}\right\},
\end{split}
\end{equation}
where $k\in\mathbb{Z}$, and $a$, $c\in\mathbb{C}[v]$.

At this point one could apply the techniques of
\cite{Bustinduy-jde} to analyze the possible global $1$-forms of
times associated to $X$ in order to prove the existence of an
invariant line. However, applying directly the local models of
\cite{Brunella-topology}, it follows from
\cite[Lemma\,2]{Bustinduy-ijm} that at least one of the
irreducible components of $R$ over $0$ must be a
$\mathcal{F}$-invariant line. Hence the polynomial $c(v)$ of
(\ref{hest}) is in fact a monomial, and thus of the form $cv^N$
with $c\in\mathbb{C}$ and $N\in\mathbb{N}$.

Finally, according to \cite[pp.\,661-662]{Bustinduy-indiana} we
know that $X_{\mid C_z}$ complete implies $k=0$ and $N=0,1$. Hence
$X$ is complete.

\section{\bf {\textnormal{kod}($\mathcal{F}$)=0}}

According to \cite[\S III and \S IV]{Mc} we can contract
$\tilde{\mathcal{F}}$-invariant rational curves on $M$ via a
contraction $s$ to obtain a new surface $\hat{M}$ (maybe singular
with cyclic quotient singularities), a reduced foliation
$\hat{\mathcal{F}}$ on this surface, and a finite covering map $r$
from a smooth compact projective surface $S$ to $\hat{M}$ such that: 1) $r$ ramifies
only over cyclic (quotient) singularities of $\hat{M}$ and 2) the
foliation $r^{\ast}(\hat{\mathcal{F}})$ is generated by a complete
holomorphic vector field $Z_{0}$ on $S$ with isolated zeroes \cite
[p.\,443]{Brunella-topology}.
$$
\xymatrix{  \mathbb{CP}^2 & \ar[l]_{\pi} M  \ar[d]^{s} &  \\
                          &    \hat{M}                 & \ar[l]_{r} S}
$$

\begin{remark}\label{levantada} \em Note that $\tilde{C}_z$ does not meet
the exceptional divisor of the contraction $s$. Let us set
$\hat{C}_z$ as  $s(\tilde{C}_z)$. Since $\hat{C}_z$ does not
contain singularities of $\hat{M}$ then $\hat{C}_z$ is a Riemann
Surface, $s_{\mid \tilde{C}_z}: \tilde{C}_z \to \hat{C}_z$ is a
biholomorphism and $r_{\mid r^{-1}(\hat{C}_z)}: r^{-1}(\hat{C}_z)
\to \hat{C}_z$ is a non-ramified finite covering map. Thus
$s_{\ast}( \tilde{X}_{\mid \widetilde{C}_z})$ is complete on
$\hat{C}_z$ and $r^{\ast} (s_{\ast}( \tilde{X}_{\mid
\tilde{C}_z}))$ is complete on the connected components
$\mathscr{M}_i$ of $r^{-1}(\hat{C}_z)=\cup_{i=0}^{l}
\mathscr{M}_{i}$. Hence each $\mathscr{M}_i$ is a Riemann Surface
contained in a complete trajectory $T_{z}$ of $Z_0$  that supports the
complete vector field $r^{\ast} (s_{\ast}( \tilde{X}_{\mid
\tilde{C}_z}))_{\mid\mathscr{M}_{i}}$, which does not necessarily
coincide with ${Z_0}_{\mid \mathscr{M}_i}$. It is convenient to observe that if $T_{z}$ is isomorphic to $\mathbb{C}^{\ast}$ then, necessarily
$\mathscr{M}_i=T_{z}$, and the vector field $r^{\ast} (s_{\ast}( \tilde{X}_{\mid
\tilde{C}_z}))_{\mid\mathscr{M}_{i}}$ coincides with $Z_{0}$ on $T_{z}$, up to a multiplicative constant. The discrepancy between the two complete vector fields
can occur only if $T_{z}$ is isomorphic to $\mathbb{C}$, in which case $r^{\ast} (s_{\ast}( \tilde{X}_{\mid
\tilde{C}_z}))_{\mid\mathscr{M}_{i}}$ could have one (and only one) zero at some point $p=T_{z}\setminus \mathscr{M}_{i}$.

\em
\end{remark}

It follows from \cite [p.\,443]{Brunella-topology} that the
covering $r$ can be lifted to $M$ via a birational morphism $g:W
\to S$ and a ramified covering $h:W \to M$ such that $s \circ h =
r \circ g$.
$$
\xymatrix{  M\ar[d]_{s}
& \ar[l]_{h} W \ar[dl]_{s \circ h}^{r \circ g} \ar[d]^{g} \\
  \hat{M} & \ar[l]^{r} S}
$$
Let $Y$ be the lift of $Z_0$ on $W$ via $g$. Then $Y$ must be a
\em rational vector field \em on $W$ generating the foliation
$\bar{\mathcal{F}}$ given by
$g^{\ast}(r^{\ast}(\hat{\mathcal{F}}))=h^{\ast}\tilde{\mathcal{F}}$.
On the other hand, $\bar{{\mathcal{F}}}$ is also generated by the
rational vector field $\bar{X}$ on $W$ given by
$h^{\ast}\tilde{X}$. Hence there is a rational function $F$ on $W$
such that
\begin{equation}
\label{tangentes} \bar{X} = F\cdot Y.
\end{equation}

\begin{remark}
\label{segundas} \em We remark from the above construction:
\begin{enumerate}
\item [{1)}] The map $g$ is a composition of blowing-ups over a
finite set $\Theta=\{ {\theta_{i}} \}_{i=1}^{s}\subset S$ of
regular points of $Z_0$. The poles of $Y$ are in $g^{-1 }(\Theta)$
and they define a divisor $P \subset W$ invariant by
$\bar{\mathcal{F}}$. Hence $Y$ is holomorphic on $W\setminus P$.
Note that in $W\setminus P$, $Y$ has only isolated zeroes.

\item [{2)}]  Since $P$ is the exceptional divisor of $g$, $h(P)$
is the exceptional divisor of $s$ and is
$\tilde{\mathcal{F}}$-invariant. Then
$$
h_{\mid W\setminus P}: W\setminus P\to M\setminus h(P)
$$
is a regular covering map.

\item [{3)}]  Let $C_ {\theta_{i}}$ be the trajectory of $Z_{0}$
through $\theta_{i}$. $Y$ is a complete holomorphic vector field
on $W\setminus \{ g^{-1}(C_ {\theta_{i}}) \}_{i=1}^{s}$. Each
$g^{-1}(C_ {\theta_{i}})\setminus P$ is contained in a trajectory
$R_{{z}_{i}}$ of $Y$.  Let us fix one of them, say $C_
{\theta_{j}}$. Let us set $\Theta \cap C_
{\theta_{j}}={\{\theta_{j_{l}}\}}_{l=0}^{h}$ taking $j_{0}=j$.
Note that $R_{{z}_{j_{l}}}=R_{{z}_{j}}$ for any $l$. For every
${\theta}_{j_{l}}$ there is a point ${\bar{\theta}_{j_{l}}} \in P$
such that $R_{{z}_{j}} \cup \,\{\bar{\theta}_{j_{l}}\}$ defines a
separatrix of $\bar{{\mathcal{F}}}$ through
$\bar{\theta}_{j_{l}}$. Note that $\bar{\theta}_{j_{l}}$ is the
unique singular point of $\bar{{\mathcal{F}}}$ in $P$ such that
$g(\bar{\theta}_{j_{l}})=\theta_{j_{l}}$. We can take around
$\bar{\theta}_{j_{l}}$ a neighbourhood $U$ and coordinates $(z,w)$
such that $\bar{{\mathcal{F}}}$ is generated by $z
\partial/\partial z - w
\partial/\partial w$ where $(R_{{z}_{j}} \cup
\,\{\bar{\theta}_{j_{l}}\})\cap U =\{w=0\}$ and
$g^{-1}({\theta}_{j_{l}})\cap U = \{z=0\}$. As $Y$ has a pole of
order one along $\{z=0\}$, it follows that
$$
Y_{\mid {R}_{{z}_{j}}} =  \frac{\partial}{\partial z} -
\frac{w}{z} \frac{\partial}{\partial w}
$$
is not complete. However, it extends on ${R_{{z}_{j}} \cup
{\{{\bar{\theta}}_{j_{l}}\}}_{l=0}^{h}}$ as a complete vector
field  because $g$ restricted to $R_{{z}_{j}}$ extends to
$R_{{z}_{j}} \cup {\{{\bar{\theta}}_{j_{l}}\}}_{l=0}^{h}$ as a
biholomorphism onto $C_{\theta_{j}}$ and
\begin{equation}\label{restrict}
{(g_{\mid R_{{z}_{j}} \cup
{\{{\bar{\theta}}_{j_{l}}\}}_{l=0}^{h}})}^{\ast} {Z_{0}}_{\mid C_
{\theta_{j}}}= Y_{ \mid R_{{z}_{j}} \cup
{\{{\bar{\theta}}_{j_{l}}\}}_{l=0}^{h}}.
\end{equation}

\em
\end{enumerate}
\end{remark}
\subsection*{Global holomorphic vector fields \cite{Brunella-impa}}
The list of holomorphic vector fields with isolated
singularities on compact complex surfaces is well known.
In \cite[Chapter\,6]{Brunella-impa} we can find the details when the surface is projective.
In particular, for $Z_{0}$ on $S$ we have one of the following possibilities:

\noindent I) $S$ has an elliptic fibration $f:S\to B$, and $Z_0$
is a nontrivial holomorphic vector field on $S$ tangent to the
fibres of $f$. Each fibre of $f$ is a smooth elliptic curve which
can be multiple, and outside multiple fibers $f$ is a locally
trivial fibration. Moreover $Z_0$ has empty zero set.

\noindent II) $S=\mathbb{C}^2/\Lambda$ is a 2-torus and $Z_0$ is a
linear vector field on it, that is, the quotient of a constant
vector field on $\mathbb{C}^2$.

\noindent III) $S$ is a $\mathbb{CP}^1$-bundle over an elliptic
curve $\mathcal{E}$, and $Z_0$ is transverse to the fibers and
projects on $\mathcal{E}$ to a constant vector field. In this case
$Z_0$ is the suspension of $\mathcal{E}$ via the representation
$\rho: \pi_{1}(\mathcal{E})\to Aut(\mathbb{CP}^1)$ associated to
the bundle structure, and it generates a Riccati foliation without
invariant fibres and whose monodromy map is $\rho$.

\noindent IV) $S$ is a rational surface, and up to a birational
map we have $Y=\mathbb{CP}^1\times \mathbb{CP}^1$ and
$Z_0=v_{1}\oplus v_{2}$, where $v_{1}$ and $v_{2}$ are holomorphic
vector fields on $\mathbb{CP}^1$.

In the course of the proof we will consider $S$ in some cases  as a differentiable manifold with
a given Riemannian metric $g$. If $(N,g)$ is a Riemannian manifold,
we denote by $d$ the distance given by the metric, and by $B^{d}_{r}(p)$ the open ball centered at $p$.
For the basic notions of Riemannian geometry that we will be used in the rest of the paper, see \cite{Petersen}.

We will analyze the possible cases for $Z_0$ and $S$. First note
that Case I) does not really occur since $\mathcal{F}$ has not
rational first integral.
\subsection{Cases II) and III)}

\begin{proposition}
\label{Proposicion1}
If $Z_0$ and $S$ are as in II) or III) any
irreducible component of $D\cup E$ is invariant by
$\tilde{\mathcal{F}}$.
\end{proposition}

\begin{proof}
Let ${D}_{0}$ be an irreducible component of $D\cup E$ that
is not invariant by $\tilde{\mathcal{F}}$. There is a compact
curve ${Q}_{0}$ (possibly singular) in $S$ generically transversal to
$Z_0$. It is enough
to define ${Q}_0$ as one of the connected components of $
r^{-1}(s({D}_0))$. Note that $s({D}_0)$ is not a point.

\vspace{0.25cm}

\noindent {\bf Case II)}. Let us take $S$ as the quotient manifold
$\mathbb{C}^2/\Lambda$. We identify $\mathbb{C}^{2}$ with  $\mathbb{R}^{4}$, and $\Lambda$ is an integral lattice of rank four.

\begin{remark}\label{toro}
\em
Let $\mu: \mathbb{C}^2 \to \mathbb{C}^2/\Lambda$ denote the
canonical submersion map. If we consider $\mathbb{R}^{4}$ with the
usual euclidean $g$, taking $g'=\mu_{\ast}g$ as the metric on $\mathbb{C}^{2}/\Lambda$, the map $\mu$ becomes a
Riemannian covering map. We will denote  by $d$ and $d'$
the distances in $(\mathbb{R}^4,g)$ and
$(\mathbb{R}^4/\Lambda,g')$, respectively. \em
\end{remark}
The vector field $Z_0$ is the projection by $\mu$ of a constant
vector field on $\mathbb{C}^2$, and thus its trajectories must be
of the form $\mu(L_t)$ where $\{L_t\}_{t\in\mathbb{C}}$ is the
family of lines parallel to a given direction. Note that \em $Z_0$
has not singularities. \em

\begin{lemma}\label{lema3}
There is a compact $K\subsetneq S$ such that
${Q}_{0}\subset\mathring{K}$
\end{lemma}
\begin{proof}
Since $Z_0$ is complete and without singularities we can define
\begin{equation}
\label{k} K=\{\varphi(T,z)\,|\,|T|\leq{1},\,z\in {Q}_0\}
\end{equation}
where $\varphi: \mathbb{C}\times S \to S$ is the complex flow of
$Z_0$. If we apply the Flow Box Theorem to the points of
${Q}_0$ we easily deduce that
${Q}_{0}\subset\mathring{K}$.
\end{proof}
We define the following
function
\begin{equation}
\label{alfa}
\begin{split}
\alpha:\,\, & \mathbb{C}\to[0,+\infty)\\
   & t\mapsto \,\,\,d(L_{t},\mu^{-1}({Q}_{0}))
\end{split}
\end{equation}
\begin{remark}\label{funcionalfa}
\em  $\alpha$ \em is continuous. \em For any sequence
$\{t_n\}_{n\in\mathbb{N}}\subset\mathbb{C}$ converging to
$t_{\ast}$ as $n\to\infty$, one sees that  $\alpha(t_n)\leq
d(L_{t_n},L_{t_{\ast}})+ \alpha(t_{\ast})$ and  $\alpha(t_{\ast
})\leq d(L_{t_{\ast}},L_{t_{n}})+ \alpha(t_{n})$. Then, $\lim_{\,n\to\infty}
\alpha(t_n)=\alpha(t_{\ast})$.
\end{remark}
\noindent We  will use that $\alpha$ has the following property
with respect to $K$.
\begin{lemma}
\label{propiedadalfa} $\alpha(t)\neq{0}$ if and only if
$\mu(L_t)\cap\mathring{K}=\emptyset$
\end{lemma}
\begin{proof}
If $\alpha(t)\neq{0}$, it is clear from (\ref{k}) that
$\mu(L_t)\cap\mathring{K}=\emptyset$. Otherwise
$\mu(L_t)\cap {Q}_{0}\neq \emptyset$, which is not possible
with our assumptions.

If $\alpha(t)=0$, we suppose $\mu(L_t)\cap\mathring{K}=\emptyset$
and obtain a contradiction.

\vspace{0.15cm}

\noindent \em Fact 1.  There is $\delta\in\mathbb{R}^{+}$ such
that $d'(\mu(L_t),{Q}_{0})\geq{\delta}$. \em

\noindent Otherwise we can determine a sequence $\{x_{n}\}_{n\in
\mathbb{N}}\subset {Q}_{0}$ converging to $x_{\ast}\in
{Q}_{0}$ and such that $d'(\mu(L_t),x_n)<1/n$ because
${Q}_{0}$ is compact and
$(\mathbb{R}^4/\Lambda,g')$ is complete. But it implies that for any ball
$\mathbb{B}^{d'}_{r}(x_{\ast})$ there exists $n(r)\in\mathbb{N}^{+}$ such
that $\mathbb{B}^{d'}_{1/n(r)}(x_{n(r)})\subset
\mathbb{B}^{d'}_{r}(x_{\ast})$, and hence
$\mathbb{B}^{d'}_{r}(x_{\ast})\cap \mu(L_t)\neq \emptyset$, which
contradicts our assumption
$\mu(L_t)\cap {Q}_{0}=\emptyset$.

\vspace{0.2cm}
\noindent \em Fact 2.
$d({L}_t,\mu^{-1}({Q}_{0}))\geq{\delta}$. \em

\noindent By contradiction, suppose that $d({L}_t,\mu^{-1}({Q}_{0}))<{\delta}$. Then there are
$z\in {L}_t$ and $\bar{z}\in\mu^{-1}({Q}_{0})$ with
$d(z,\bar{z})<{\delta}$. Note that $\mu^{-1}({Q}_{0})$ is
an analytic variety (non necessarily compact) of $\mathbb{C}^2$
and that $\mu(z)\neq \mu(\bar{z})$ by Fact~1. Let $c$ be a
segment from $z$ to $\bar{z}$. As $\mu$ defines a local isometry
from $(\mathbb{R}^4,g)$ to $(\mathbb{R}^4/\Lambda,g')$, we can
take $\mathbb{B}^{d}_{r_{i}}(z_{i})\subset \mathbb{R}^4$, $0\leq i
\leq s$, centered at $z_i\in c$, where $z_0=z$ and $z_s=\bar{z}$,
and in such a way that $\mu$ restricted to each
$\mathbb{B}^{d}_{r_{i}}(z_{i})$ defines an isometry over its
image. Moreover, we can assume that
$\mathbb{B}^{d}_{r_{i}}(z_{i})\cap\mathbb{B}^{d}_{r_{j}}(z_{j})\neq
\emptyset$ if and only if $j=i+1$, and thus fix $s-1$ points
$z_{i,i+1}$ in these  intersections. As the isometries preserve
intrinsic distance, Fact~1 and the triangle inequality implies the
following contradiction
\begin{equation*}
\begin{split}
\delta >  d(z,\bar{z})= & \sum_{i=1}^{s-1} d(z_{i},z_{i,i+1}) + d(z_{i,i+1},z_{i+1}) =\\
                        & \sum_{i=1}^{s-1}  d'(\mu(z_{i}),\mu(z_{i,i+1})) + d'(\mu(z_{i,i+1}),\mu(z_{i+1})) \geq \\
                        &  d'(\mu(z),\mu(\bar{z}))\geq \delta
\end{split}
\end{equation*} \end{proof}
\begin{lemma}\label{alfa0}
$\mu(L_{t})\cap {Q}_{0}\neq{\emptyset}$ for any
$t\in\mathbb{C}$, and then $\alpha\equiv{0}$.
\end{lemma}
\begin{proof}
Suppose, by contradiction, that $\mu(L_{t})\cap {Q}_{0}=
\emptyset$. It implies that
$\mu(L_{t})\cap\mathring{K}=\emptyset$. On the other hand, by
Lemma~\ref{propiedadalfa}, $\alpha(t)\neq{0}$. Then
$\alpha^{-1}(0)$ is a closed set strictly contained in
$\mathbb{C}$, and if we take $\tilde{t}$ on its boundary we can fix a sequence $\{t_{n}\}_{n\in\mathbb{N}}$
with $\alpha(t_{n})\neq{0}$ converging to $\tilde{t}\in\mathbb{C}$
with $\alpha(\tilde{t})=0$. Note that
$\mu(L_{\tilde{t}})\cap {Q}_{0}\neq\emptyset$ due to
$\mu(L_{\tilde{t}})\cap\mathring{K}\neq\emptyset$ since
$\alpha(\tilde{t})=0$ (Lemma~\ref{propiedadalfa}).

Let us take $\tilde{x}\in \mu(L_{\tilde{t}})\cap {Q}_{0}$
with $\mu(\tilde{z})=\tilde{x}$, and set
$\{z_n\}_{n\in\mathbb{N}}$ converging to $\tilde{z}$ with $z_n\in
L_{t_n}$. By continuity, $\{\mu(z_n)\}_{n\in\mathbb{N}}$ must
converge to $\tilde{x}$. However, as $z_n\in L_{t_n}$ for any $n$,
it holds $\mu(z_n) \notin \mu(L_{t_n})\cap\mathring{K}$ since
$\alpha(t_n)\neq{0}$ (Lemma~\ref{propiedadalfa}), what is a
contradiction. Then $\mu(L_{t})\cap
{Q}_{0}\neq{\emptyset}$.
\end{proof}
It follows from Remark~\ref{levantada} that $\mathscr{M}_i$ is
contained in a trajectory of $Z_0$. Hence there is $L_{s_i}$ such
that $\mu(L_{s_i})\supset\mathscr{M}_i$.
\begin{lemma}\label{ceros}
$\mu(L_{s_i})\cap {Q}_{0}=\{p_i\}$, where $p_{i}$ is the unique point in
$\mu(L_{s_i})\setminus\mathscr{M}_i$. In particular, $\mu(L_{s_i})$ and $\mathscr{M}_i$ are respectively biholomorphic to $\mathbb{C}$ and
$\mathbb{C}^{\ast}$.
\end{lemma}
\begin{proof} Lemma~\ref{alfa0} implies that $\mu(L_{s_i})\cap
{Q}_{0}\neq{\emptyset}$. Moreover it is clear that
$\mu(L_{s_i})\cap
{Q}_{0}\subset\mu(L_{s_i})\setminus\mathscr{M}_i$. It follows from
Remark~\ref{levantada} that $r^{\ast} (s_{\ast}(\tilde{X}_{\mid
\tilde{C}_z}))_{\mid\mathscr{M}_{i}}$ is complete and then
it extends holomorphically
by zeroes on $\mu(L_{s_i})\setminus\mathscr{M}_i$. Since $C_{z}$ is not algebraic,
$\mathscr{M}_i$ is biholomorphic to $\mathbb{C}^{\ast}$ and  $\mu(L_{s_i})\cap {Q}_{0}=\mu(L_{s_i})\setminus\mathscr{M}_i$ is an unique point $p_{i}$. Thus $\mu(L_{s_i})$ must be biholomorphic to $\mathbb{C}$.
\end{proof}

Since the foliation defined by $Z_{0}$ on  $S$ has codimension $1$ and it has not singularities,
the closure of $\mu(L_{s_i})$ in the open set $U'\subset S$ of non-compact leaves, that we will denote by $L'$, is a subvariety of real codimension $0$, $1$ or $2$ \cite[Th\'eor\`eme 1.4]{Ghys}. \em It holds $U'=S$ and then $L'$ is the closure of $\mu(L_{s_i})$ in $S$. \em If there were one compact leaf $J$, \cite[Th\'eor\`eme 1.4]{Ghys} also assures that any non-compact leaf must accumulate $J$. In particular $\mu(L_{s_i})$ accumulates $J$. On the other side, as ${Q}_{0}$ cut any leaf (Lemma~\ref{alfa0}), it must cut $J$, and $\mu(L_{s_i})$ accumulates the points of $J\cap {Q}_{0}$, which is not possible since $p_{i}$ is the unique point in $\mu(L_{s_i})\setminus\mathscr{M}_i$ (Lemma~\ref{ceros}). Note that \em $L'$ is $S$ or a real compact subvariety of dimension three.
\em If $L'$ had real codimension $2$, it would define a real compact subvariety of dimension two of $S$ ($L'$ is closed in $S$) containing $\mu(L_{s_i})$, which is a non-algebraic leaf. One concludes that $\mu(L_{s_i})$ must intersect infinitely many times ${Q}_{0}$, and then one obtains again a contradiction with Lemma~\ref{ceros}.

\vspace{0.15cm}

\noindent {\bf Case III)}. Let us consider $S$ as
$\mathbb{CP}^1$-bundle over an elliptic curve $\mathcal{E}$ with
bundle projection $p:S\to \mathcal{E}$. The structure of $S$  can
be lifted as $\mathbb{CP}^1$-bundle $\tilde{S}$ over $\mathbb{C}$
via the universal covering map $\Gamma:\mathbb{C}\to\mathcal{E}$:
we can determine a complex surface $\tilde{S}$, a holomorphic
covering $F:\tilde{S}\to S$ and a bundle projection
$\tilde{p}:\tilde{S}\to\mathbb{C}$ such that $p\circ F= \Gamma
\circ \tilde{p}$.
$$
\xymatrix{   \tilde{S} \ar[d]^{\tilde{p}}  \ar[r]^{F}  &  S  \ar[d]^{p} &  \\
                 \mathbb{C}   \ar[r]_{\Gamma}             &    \mathcal{E}     }
$$
Moreover as $\mathbb{C}$ is contractible this $\mathbb{CP}^1$-bundle is trivial. Thus
$\tilde{S}=\mathbb{C}\times\mathbb{CP}^1$ and $\tilde{p}(x,y)=x$ is the projection over the first factor.

\begin{lemma}\label{horizontal}
There is a holomorphic automorphism $\sigma$ of
$\mathbb{C}\times\mathbb{CP}^1$ such that $\sigma^{\ast}(F^{\ast}Z_0)$
is the horizontal vector field.
\end{lemma}
\begin{proof}
It is clear that $F^{\ast}Z_0$ generates a Riccati foliation adapted to $\tilde{p}$ and without invariant fibres.
If $\sigma(t,y)=\tilde{\varphi}(t,0,y),$
with $\tilde{\varphi}$ the complex flow of $F^{\ast}Z_0$,
$\sigma$ is bijective, since each trajectory of $F^{\ast}Z_0$ intersects each fibre of $p$ in only one point,
and $\sigma(\mathbb{C}\times \{y\})$ are the trajectories of $F^{\ast}Z_0$.
\end{proof}

After Lemma~\ref{horizontal}, the trajectories of $Z_0$ are of the
form $(F\circ \sigma)(L_t)$ where $\{L_t\}_{t\in\mathbb{CP}^{1}}$
is now the family of lines $L_{t}= \mathbb{C}\times\{t\}$.

\begin{remark}\label{metrica}
\em As $S$ is compact, $S$ (as real manifold) admits a Riemannian
metric $g'$. Let us set
$\bar{\mu}=F \circ \sigma$. The map $\bar{\mu}$
from $(\mathbb{R}^2 \times \mathbb{S}^2, {\bar{\mu}}^{\ast} g')$
to $(S,g')$ is a local Riemannian isometry. But still more, $F$ is
a covering map and $\sigma$ is a biholomorphism, hence
$F \circ \sigma$ is also a covering map and $\bar{\mu}$ is a
Riemannian covering map. As $(S,g')$ is compact, it is complete,
and  $(\mathbb{R}^2 \times
\mathbb{S}^2, {\bar{\mu}}^{\ast} g')$ is complete. We will denote
by $d$ and $d'$ the distances in $(\mathbb{R}^2 \times
\mathbb{S}^2, {\bar{\mu}}^{\ast} g')$ and $(S,g')$, respectively.
\em
\end{remark}
\em The vector field $Z_0$ is complete and without zeroes. \em We
will consider as in case II) an irreducible component ${D}_{0}$ of
$D\cup E$ that is not invariant by
$\tilde{\mathcal{F}}$, and the compact curve $Q_{0}$
(possibly singular) in $S$ generically transversal to $Z_0$,
defined by one of the connected components of $r^{-1}(s({D}_0))$.
As in Lemma~\ref{lema3} we can determine a compact set
$K\subsetneq S$ such that $Q_0 \subset \mathring{K}$.

We will consider the continuous map (it follows as in Remark 7)
\begin{equation}
\label{alfabar}
\begin{split}
\bar{\alpha}:\,\, & \mathbb{CP}^{1}\to[0,+\infty)\\
   & t\mapsto \,\,\,d(L_{t},\bar{\mu}^{-1}({Q}_{0}))
\end{split}
\end{equation}

Once we have fixed $\bar{\mu}=F \circ \sigma$, the complete
metrics in Remark~\ref{metrica}, the compact set $K$ and the map
$\bar{\alpha}$ as (\ref{alfabar}), we can prove similar Lemmas to
Lemmas~\ref{lema3}, \ref{propiedadalfa}, \ref{alfa0} and
\ref{ceros} of Case III), where  $\mu$ and $\alpha$ must be
substituted by $\bar{\mu}$ and $\bar{\alpha}$ in their statements.

Let $L_{s_i}$ be such that
$\bar{\mu}(L_{s_i})\supset\mathscr{M}_i$. Since  $p_{\mid
\,\,\bar{\mu}(L_{s_i})}: \bar{\mu}(L_{s_i}) \to \mathcal{E}$ is a
covering map, and $\bar{\mu}(L_{s_i})$ is biholomorphic to
$\mathbb{C}$ (Lemma\,\ref{ceros}), \em $\bar{\mu}(L_{s_i})$ must
cut almost all the fibres of $p$ infinitely many times. \em Let
$\kappa\in\mathcal{E}$ such that $p^{-1}(\kappa)$ contains an
infinite sequence of different points in
$p^{-1}(\kappa)\cap\bar{\mu}(L_{s_i})$. By compactness of
$p^{-1}(\kappa)$, the above sequence converges to  $q_{1}\in
p^{-1}(\kappa)$. Note that $q_{1}$ is a regular point of $Z_{0}$.
If $\bar{\mu}(L_{\tilde{s}})$ is the trajectory of $Z_{0}$ through
$q_{1}$, $\bar{\mu}(L_{s_i})$ must accumulate
$\bar{\mu}(L_{\tilde{s}})$ (flow-box theorem). On the other hand,
$\bar{\mu}(L_{\tilde{s}})\cap {Q}_{0}\neq\emptyset$
(Lemma~\ref{alfa0}) implies a contradiction with  the fact that
$p_{i}$ is the unique point in
$\bar{\mu}(L_{s_i})\setminus\mathscr{M}_i$ (Lemma~\ref{ceros}).

\begin{remark}
\em One can also obtain a contradiction by distinguishing several
cases, according to the (abelian) monodromy $\Gamma\subset
Aut(\mathbb{CP}^{1})$. If $\Gamma$ has rank 1, then the
non-algebraic leaves of $Z_{0}$ are isomorphic to
$\mathbb{C}^{\ast}$, and one gets a contradiction by using
Remark~5, and the fact that the intersection with algebraic curves
is nonempty. If $\Gamma$ has rank $2$ then $\Gamma=<f,g>$ with
$f(z)=\alpha z$, $g(z)=\beta z$ or $f(z)=z+1$, $g(z)=z+w$. In the
first case (where, moreover, ${\alpha}^{n}{\beta}^{m}\neq {1}$ for
every $(m,n)\neq(0,0)$) the non-algebraic leaves are sufficiently
dense to apply the same argument as in case II. In the second case
one can prove that every algebraic curve $C\subset S$ different
from the elliptic curve $E=\{z=\infty\}$ must intersect $E$, and
from this fact it follows again that every non-algebraic
trajectory of $Z_{0}$ intersects $C$ infinitely many times. \em
\end{remark}
\end{proof}
\subsection{Case IV)}
There is a birational transformation $G:
S\rightarrow\mathbb{CP}^1\times\mathbb{CP}^1$ sending $Z_0$ to
$G_{\ast} Z_{0}=v_{1}\oplus v_{2}$, where $v_{1}$ and $v_{2}$ are
holomorphic vector fields on $\mathbb{CP}^1$. The description of
$G$ can be found in \cite[p.\,87]{Brunella-impa}. In particular,
$G$ is a finite sequence of birational transformations which are
contractions of curves invariant by $Z_0$ or blowing-ups at zeros
of $Z_0$. Hence the exceptional divisor of $G$ does not meet
$\mathscr{M}_i$, and as consequence $G(\mathscr{M}_i)$ is
biholomorphic to $\mathscr{M}_i$. But still more, as
$\mathscr{M}_i$ supports a complete vector field according to
Remark~\ref{levantada}, we can define an entire curve
$f:\mathbb{C}\to G(\mathscr{M}_i)$. In absence of rational first
integrals, we may assume that  $G_{\ast} Z_0$ is not constant. \em
Note that $G(\mathscr{M}_i)$ is contained in a trajectory $L_{z}$ of
$G_{\ast} Z_0$ and that $L_{z}\setminus G(\mathscr{M}_i)$ is empty or
one point. \em There are two cases for $G_{\ast} Z_0$:

\noindent {\bf a)} \em  $v_{1}$ and $v_{2}$ with zeroes of order
one at $0$ $(\lambda z \partial / \partial z + \mu w \partial / \partial w)$. \em \\
\noindent {\bf b)} \em $v_{1}$ with a zero of order one at $0$
and $v_{2}$  constant $(\lambda z \partial / \partial z + \mu  \partial / \partial w)$. \em

%Otherwise Theorem~1 follows by the results of $\S 2$.

\begin{proposition}
\label{proposicion2} There exists at most an irreducible component
$D_{0}$ of $D \cup E$ that is not invariant by
$\tilde{\mathcal{F}}$. If $D_{0}$ exists:

\noindent 1) $G_{\ast} Z_0$ is as in \em{\bf b)};\em \\
\noindent 2) $r^{-1}(s({D}_{0}))= {Q}_0$; \\
\noindent 3) The strict transform $A_{0}$ of ${Q}_0$ by $G$ is $\overline{\{w=0\}}$; and \\
\noindent 4) $\tilde{\mathcal{F}}$ is a Riccati foliation adapted
to a rational map that projects by $\pi$ as a rational function
$R$ of type $\mathbb{C}$ or $\mathbb{C}^{\ast}$.
\end{proposition}
\begin{proof}
Let  ${D}_0$ be a component of $D\cup E$  not invariant by $\tilde{\mathcal{F}}$, and $Q_{0}$ be
the curve in $S$ defined as in Proposition~\ref{Proposicion1}. If $G_{\ast} Z_0$ is as {\bf b)} let
us suppose  that  either there is one component $D_{j}$ of ${D}\cup {E}$ not
invariant by $\tilde{\mathcal{F}}$ and different from $D_{0}$ or there is other component of
$r^{-1}(s({D}_{0}))$ different from ${Q}_{0}$.
\begin{lemma}
\label{lema8} There exist an open set $B\subset
\mathbb{CP}^1\times\mathbb{CP}^1$ biholomorphic to $\mathbb{C}^2$
and an entire curve $\bar{f}:\mathbb{C}\to G(\mathscr{M}_i)\cap B$
tangent to ${G_{\ast} Z_0}_{\mid B}$ whose image avoids at least
three algebraic curves contained in $B$.
\end{lemma}
\begin{proof} We analyze the two cases:

\noindent $G_{\ast} Z_0$ as in {\bf a)}: Let $B$ be $\mathbb{CP}^1
\times \mathbb{CP}^1$ minus $\{z=\infty\}\cup\{w=\infty\}$. As
$\{z=\infty\}$ and $\{w=\infty\}$ are invariant by ${G_{\ast}
Z_0}$, $G(\mathscr{M}_i)\subset  B$. Note that ${G_{\ast} Z_0}$ on
$B$ is complete. If  $\bar {f}=f$, $\bar {f}:\mathbb{C}\to
G(\mathscr{M}_i)\subset B$  is an entire map whose image avoids at
least $\{z=0\}$, $\{w=0\}$ and $A_{0}\cap B$.

\noindent $G_{\ast} Z_0$ is as in {\bf b)}: Let $B$ be
$\mathbb{CP}^1 \times \mathbb{CP}^1$ minus $\{z=0\}\cup\{w=0\}$.
In this case $\{z=0\}$ is invariant by ${G_{\ast} Z_0}$ but
$\{w=0\}$ is not. Remark that any non-algebraic trajectory of
${G_{\ast} Z_0}$ is of type $\mathbb{C}$ and intersects $\{w=c\}$,
with $c\neq\infty$, in an unique point. More still, \em one can
suppose that $A_{0}\neq \overline{\{w=0\}}$. \em Otherwise
one define ${Q}_0$ as any other component of
$r^{-1}(s({D}_{0}))$ or $r^{-1}(s({D}_{j}))$, where  $D_{j}$ is a
component of ${D}\cup {E}$ not invariant by
$\tilde{\mathcal{F}}$  and $D_{j}\neq D_{0}$.

\noindent b.1) If $L_{z}\setminus G(\mathscr{M}_i)=\emptyset$, we take
$G(\mathscr{M}_i)\cap \{w=0\} =p$ and the trajectory
$G(\mathscr{M}_i)\cap B=L_{z}\setminus\{p\}\simeq\mathbb{C}^{\ast}$ of
${G_{\ast} Z_0}$ on $B$. As the universal covering of
$G(\mathscr{M}_i)\cap B$ is $\mathbb{C}$, there exists
$\bar{f}:\mathbb{C}\to G(\mathscr{M}_i)\cap B$ whose image
avoids at least the algebraic curves: $\{z=\infty\}$,
$\{w=\infty\}$ and $A_{0}\cap B$.

\noindent b.2) If $L_{z}\setminus G(\mathscr{M}_i)=q\in \{w=0\}$, the
argumentation is similar to b.1) since $G(\mathscr{M}_i)\cap
B=G(\mathscr{M}_i)= L_{z}\setminus\{q\}\simeq\mathbb{C}^{\ast}$ is a
trajectory of ${G_{\ast} Z_0}$ on $B$.

\noindent b.3) If $L_{z}\setminus G(\mathscr{M}_i)=q \not\in \{w=0\}$,
we take the automorphism of $\mathbb{CP}^1\times\mathbb{CP}^1$,
$(z,w)\mapsto \delta(z,w)=(z, w-q_2)$, where $q=(q_1,q_2)$. As
$\delta$ leaves invariant ${G_{\ast} Z_0}$ since
$\delta_{\ast}{G_{\ast} Z_0}={G_{\ast} Z_0}$ and $\delta(q)\in
\{w=0\}$, it is enough to apply b.2) to
$L_{z}\setminus\delta(G(\mathscr{M}_i))=\{\delta(q)\}$.
\end{proof}
Let $\mathbb{CP}^2$ be the compactification of $B$. The image of
$\bar{f}:\mathbb{C}\to G(\mathscr{M}_i)\cap B$ is contained in
$\mathbb{CP}^2$ minus at least four hypersurface sections, that
is, three sections defined by the algebraic curves of
Lemma~\ref{lema8} along with the line at infinity
$\mathbb{CP}^2\setminus B$. According to Green's Theorem
\cite[p.\,199]{Lang}, $\bar{f}(\mathbb{C})$ must be contained in
some algebraic curve, what contradicts our assumptions. Hence $1),
2)$ and $3)$ of the statement of Proposition follows.

Note that $C_{z}$, $\tilde{C}_{z}$, $\hat{C}_{z}$,
$G(\mathscr{M}_i)$ and $\mathscr{M}_i$ are biholomorphic to
${\mathbb{C}}^{\ast}$, and that $L_{z}\simeq\mathbb{C}$ and
$G(\mathscr{M}_i)= L_{z}\setminus\{q\}\simeq\mathbb{C}^{\ast}$
with $q\in\{w=0\}\setminus\{(0,0),(\infty,0)\}$.
$G(\mathscr{M}_i)$ has two parabolic ends, which are properly
embedded in the complementary set of
$\{z=\infty\}\cup\{w=\infty\}$ in $\mathbb{CP}^1 \times
\mathbb{CP}^1$. One ${\Sigma}_{1}$ defined by a punctured disk
centered at $q$, that is algebraic; and other ${\Sigma}_{2}$
defined by $G(\mathscr{M}_i)\setminus{\Sigma}_{1}$, that is
transcendental and accumulates $\{w=\infty\}$. Note that $G_{\ast}
Z_0$ has two saddle-nodes as singularities. One at $(0,\infty)$,
with strong separatrix inside $\{w=\infty\}$ and weak separatrix
inside $\{z=0\}$; and other one at $(\infty,\infty)$, with strong
separatrix inside $\{w=\infty\}$ and weak separatrix inside
$\{z=\infty\}$. On the other hand $G_{\ast} Z_0$ defines a Riccati
foliation adapted to $\beta(z,w)=w$. One may  assume (maybe after
blowing-up reduced singularities) that $\bar{\mathcal{F}}$ is
Riccati with respect to $\beta_{W}=\beta\circ G \circ g$ and that
$G\circ g$ is the contraction of curves inside fibers of
$\beta_{W}$  that produces the local models of
\cite[p.\,439]{Brunella-topology}. In this case all the fibers
$\{w=c\}$, with $c\neq \infty$, are transversal minus one that is
tangent, $\{w=\infty \}$, and of class $(d)$.

Since $h$ is an algebraic covering map from $W$ to $M$, the proper
mapping theorem allows to define the trace of $\beta_{W}$ as a
rational function $\beta_{M}$ on $M$ \cite{Gr}.
Moreover, one can assume that $\beta_{M}$ is a fibration after
eliminating its base points. Recall that the property of being reduced is stable by blowing ups. Moreover,
the possible dicritical components of the resolution of the pencil given by
$\beta_{M}$ must be transversal to the corresponding foliation.

By construction, the generic fiber $F$ of
$\beta_{M}$ is a curve transverse to $\tilde{\mathcal{F}}$. Note
that $D_{0}$ must be contained in a fiber $F_{0}$ of $\beta_{M}$
as a consequence of $3)$ in the statement of this Proposition. Let us
consider the following cases according to the genus of $F$.

\noindent $\bullet$ If $F$ is of genus $\geq 2$, it follows from
\cite[Theorem III.6.1]{PS} that $\tilde{\mathcal{F}}$ has a
rational first integral, which is not possible.\\
\noindent $\bullet$ If $F$ is of genus $1$, $\tilde{\mathcal{F}}$
is a Turbulent foliation. Let us see that this case neither occurs
because it would imply the existence of a rational first
integral as before. Indeed, note that $F$ does not cut $F_{0}$ since $\beta_{M}$ is
a fibration. On the other hand, $D\cap F\neq \emptyset$ by the
maximum principle. As $D_{0}$ is the unique irreducible component
of $D\cup E$ that is not invariant by $\tilde{\mathcal{F}}$, there
must be one $\tilde{\mathcal{F}}$-\,invariant component $D_{2}$ of
$D$ such that $D_{2}\cap F\neq \emptyset$. The existence of
$D_{2}$ implies that $\tilde{\mathcal{F}}$ has a first integral
(Lemma~1).\\
\noindent $\bullet$ If $F$ is of genus $0$, $\tilde{\mathcal{F}}$
is a  Riccati foliation. Let us see that $\tilde{\mathcal{F}}$
satisfies $4)$ of the statement of Proposition 2. After
contraction of rational curves each fiber of $\beta_{M}$ admits
one of the five possible models in \cite[\S\,7]{Brunella-ensp},
\cite[p.\,439]{Brunella-topology}. If there is one fiber ${F}_{1}$
tangent to $\tilde{\mathcal{F}}$, as $D_{0}$ is the unique
irreducible component of $D\cup E$ which is not invariant by
$\tilde{\mathcal{F}}$ and it is contained in $F_{0}$, then $F$
must cut $D$ in only one or two points near $F_{0}$. Then we can
conclude as in Lemma~2 that $\beta_{M}$ projects  by $\pi$ as a
rational function $R$ of type $\mathbb{C}$ or $\mathbb{C}^{\ast}$.
Finally, one shows that all the fibers of $\beta_{M}$ are not
transverse to $F$. In the contrary case, if $L_{0}$ is the leaf
defined by $\tilde{C}_{z}$, as the covering map
${{\beta}_{M}}_{|L}:L\to \mathbb{CP}^{1}$ is not finite (otherwise
$L$ is compact and $C_{z}$ is algebraic), $L$ must cut infinitely
many times $F_{0}$ and $C_{z}$ is not isomorphic to
$\mathbb{C}^{\ast}$, which is not possible.
\end{proof}
After Proposition 2 we can assume that any irreducible component
of ${D}\cup {E}$ is invariant by $\tilde{\mathcal{F}}$. Otherwise
Theorem~1 follows by the results of $\S 2$.

\subsection{Existence of a second integral}
Let us come back to the beginning of $\S 3$, and consider
(\ref{tangentes}).

\begin{lemma}
\label{lema9} It holds $Y^2F=0$. In particular $\bar{X}$ is
complete on the Zariski open set $W'$ of $W$ defined by $W
\setminus (\{F=0\} \cup \{F=\infty\}\cup P)$
\end{lemma}
\begin{proof}
Let $\mathscr{R}_{0}$ be a connected component of
$h^{-1}(\tilde{C}_z)$. As ${\mathscr{R}}_{0}$ does not meet the
exceptional divisor of $s:M\to \hat{M}$ then $h_{\mid {\mathscr{R}}_0}:
{\mathscr{R}}_0\to \tilde{C}_z$ is a non-ramified finite covering
map. Hence $h_{\mid \mathscr{R}_{0}}^{\ast} (\tilde{X}_{\mid
\tilde{C}_z})={\bar{X}}_{\mid \mathscr{R}_{0}}$ is complete. On
the other hand $\bar{X}$ and $Y$ are tangent on $ \mathscr{R}_{0}$
according to (\ref{tangentes}). Thus $\mathscr{R}_{0}$ is a
Riemann Surface contained in a trajectory $R_z$ of $Y$. Let
$\varphi_z:\Omega_z\to {R}_{z}$ be the corresponding solution.  We
have two possibilities from 3.- of Remark~\ref{segundas}:
\begin{enumerate}
\item[{i)}] $R_z\not \in {\{ R_{{z}_{i}}\}}_{i=1}^{s}$. Since
$Y_{\mid {R}_{z}}$ is complete $\Omega_z=\mathbb{C}$.

\item[{ii)}]$R_z\in {\{ R_{{z}_{i}}\}}_{i=1}^{s}$. Let us suppose
that $R_z= R_{{z}_{j}}$. We take the solution
$f_{{z}_{j}}:\mathbb{C}\to R_{{z}_{j}} \cup
{\{{\bar{\theta}}_{j_{l}}\}}_{l=0}^{h}$ of (\ref{restrict}) and
the discrete subset $\Delta =
{\{{f_{{z}_{j}}}^{-1}(\bar{\theta}_{j_{l}})\}}_{l=0}^{h}$ of
$\mathbb{C}$. Since ${f_{{z}_{j}}}_{\mid
\mathbb{C}\setminus\Delta}= \varphi_z$ then
$\Omega_z=\mathbb{C}\setminus\Delta$.
\end{enumerate}

It follows from i) and ii) that $\varphi_z$ is a \em univaluated
\em holomorphic map. Let us note that ${\bar{X}}_{\mid R_z}$ must
be also complete because ${\bar{X}}_{\mid \mathscr{R}_{0}}$ is
complete. Using these two facts and that $T\in\Omega_z\mapsto
\varphi_z (T)\in {R}_{z}$ is a covering map then
\begin{equation}\label{derivada}
\varphi_{z}^{\ast}({\bar{X}}_{\mid R_z})
=\varphi_{z}^{\ast}(\bar{X})(T)=(F\circ
\varphi_z(T))\cdot\varphi_{z}^{\ast}(Y)=(F\circ
\varphi_z(T))\frac{\partial}{\partial T}
\end{equation}
is a complete vector field on $\Omega_z$. What is only possible if
$\Omega_z=\mathbb{C}$ or $\mathbb{C}^{\ast}$ and
$(F\circ\varphi_z)(T)=aT+b$, for $a$, $b\in\mathbb{C}$. We
conclude that $Y(F)(\varphi_z(T))=(F\circ\varphi_z)^{'}(T)$ is
constant and hence $Y^2 F$ vanishes along $R_{z}$, which can be
assumed to be non-algebraic since $C_z$ is by hypothesis. Hence
$Y^2 F = 0$.

Let us take a point $z\in W'=W \setminus (\{F=0\} \cup
\{F=\infty\}\cup {P})$. If $S_{z}$ is the trajectory of
$\bar{X}$ through $z$, as $Y$ is holomorphic on $W'$ (1.- of
Remark~\ref{segundas}) and tangent to $\bar{X}$ on $S_z$ by
(\ref{tangentes}) then $S_z$ defines a trajectory $R_z$ of $Y$.
Since it holds (\ref{derivada}), the fact that $Y^2 F=0$ implies
that $\bar{X}_{\mid {S}_{z}}$ is complete.
\end{proof}

After Lemma~\ref{lema9}, $\bar{X}$ is complete on $W'=W \setminus
(\{F=0\} \cup \{F=\infty\}\cup P)$. According to 1.- and 2.- of
Remark~\ref{segundas},  $\tilde {X}$ is complete on $M\setminus
h(W \setminus W')$. By Propositions~\ref{Proposicion1} and
\ref{proposicion2}, $\tilde{X}$ is complete on $M\setminus
(h(W\setminus W')\cup E \cup D)$. If we project by $\pi$ we see
that $X$ is complete on a Zariski open set of $\mathbb{C}^2$ and
it can be extended to $\mathbb{C}^2$ as complete vector field.
Therefore $X$ is complete.

\bibliographystyle{plain}
\bibliographystyle{amsalpha}
%\bibliography{albib}
\def\cprime{$'$} \def\cprime{$'$}

\end{document}